\documentclass[11pt]{article}

\usepackage{authblk}

\title{On the $L^{2}$ estimates of the diffusion waves}
\author[1]{Ryo Ikehata\thanks{ikehatarmath51@gmail.com}}
\author[2]{Hiroshi Takeda\thanks{Corresponding author: h-takeda@fukuoka-u.ac.jp}}
\affil[1]{Department of Mathematics, 
Division of Educational Sciences, 
Graduate School of Humanities and Social Sciences, 
Hiroshima University,
Higashi-Hiroshima 739-8524, Japan }
\affil[2]{Department of Applied Mathematics, 
Faculty of Science, Fukuoka University, 
Nanakuma, Jonan-ku, Fukuoka 814-0180, Japan}
\date{}

\usepackage{amssymb,amsmath,amsthm}
\usepackage[dvips]{graphicx}
\usepackage{bm} 

\newtheorem{thm}{Theorem}[section]
\newtheorem{cor}[thm]{Corollary}
\newtheorem{prop}[thm]{Proposition}
\newtheorem{lem}[thm]{Lemma}
\theoremstyle{remark}

\theoremstyle{definition}

\begin{document}
\maketitle

\numberwithin{equation}{section}

\begin{abstract}
In this paper, we investigate the long-time behavior of the $L^2$-norm of solutions to the Cauchy problem for the strongly damped wave equation on $\mathbb{R}^n$, with particular focus on the low-dimensional cases $n=1$ and $n=2$. Although the energy is dissipative, the $L^2$-norm may grow because of low-frequency effects. We compare the diffusion-wave profile of the strongly damped equation with the corresponding free-wave evolution generated by the same initial velocity. Introducing the difference operator $D(t)$ between these two evolutions, we prove that in one dimension $D(t)$ is controlled by $Ct^{1/4}\|g\|_{L^1}$, showing that the free wave remains an effective asymptotic profile. In contrast, in two dimensions $D(t)$ has a logarithmic lower bound when the mass of the initial velocity is nonzero, implying that the wave approximation fails. Corresponding estimates for the original solution are also obtained.
\end{abstract}

\noindent
\textbf{Keywords: }Strongly damped wave equation, diffusion wave, grow-up \\
\noindent
\textbf{MSC2020: }Primary 35L05; Secondary 35B40, 35B45

\newpage
\section{Introduction}

In this paper, we study the global-in-time behavior of the $L^2$-norm
of solutions to the Cauchy problem for the strongly damped wave equation
\begin{equation} \label{eq:1.1}
\left\{
\begin{aligned}
& \partial_t^2 u - \Delta u - \nu \Delta \partial_t u = 0,
    && t > 0,\quad x \in \mathbb{R}^n, \\
& u(0,x) = u_0(x),\quad \partial_t u(0,x) = u_1(x),
    && x \in \mathbb{R}^n,
\end{aligned}
\right.
\end{equation}
with particular emphasis on the low-dimensional cases $n=1$ and $n=2$.
Here, $\nu>0$ is the viscosity coefficient, and $u_0$ and $u_1$ denote the
prescribed initial data. It is standard that, for any
\[
(u_0,u_1) \in H^1(\mathbb{R}^n) \times L^2(\mathbb{R}^n),
\]
there exists a unique solution
\[
u \in C([0,\infty);H^1(\mathbb{R}^n))
\cap C^1([0,\infty);L^2(\mathbb{R}^n))
\]
to \eqref{eq:1.1}; see, for instance, \cite{E}.

Our motivation for deriving sharp time-dependent estimates of
$\|u(t)\|_{L^2}$ is that such estimates reveal information that is not
captured by the standard energy method. Indeed, the total energy
associated with \eqref{eq:1.1}, defined by
\begin{equation} \label{eq:1.2}
E[u](t)
= \frac{1}{2} \int_{\mathbb{R}^n}
\bigl(|\partial_t u(t,x)|^2 + |\nabla_x u(t,x)|^2\bigr)\,dx,
\end{equation}
satisfies the dissipation identity
\begin{equation} \label{eq:1.3}
\frac{d}{dt}E[u](t)
= -\nu \int_{\mathbb{R}^n}|\nabla \partial_t u(t,x)|^2\,dx
\le 0.
\end{equation}
Although \eqref{eq:1.3} shows that the energy is non-increasing, it does
not directly control the $L^2$-norm of the solution itself. In low space
dimensions, this norm may grow in time due to the dominance of
low-frequency components. This growth can be interpreted as a
mass-accumulation phenomenon, or equivalently as a manifestation of the
diffusive structure inherent in the equation.

A precise description of the growth or decay of the $L^2$-norm is
therefore central to the study of diffusion phenomena. It is also a
fundamental ingredient in the perturbation theory for nonlinear problems.
In the analysis of semilinear and quasilinear damped wave equations,
sharp linear estimates are indispensable for controlling Duhamel terms;
if the available $L^2$-bounds are too coarse, the corresponding fixed
point argument cannot be closed in the desired function spaces.

Thus, for damped wave equations, the temporal behavior of the $L^2$-norm
is not merely an auxiliary consequence of energy decay. Rather, it is a
primary quantity that characterizes the interplay between wave propagation
and diffusion, especially in the low-frequency regime.

In this setting, it is well known that solutions may exhibit the so-called
grow-up phenomenon. More precisely, if
\[
m_1 := \int_{\mathbb{R}^n} u_1(x)\,dx \ne 0,
\]
then, for sufficiently large $t$, the solution satisfies
\begin{equation} \label{eq:1.4}
C |m_1| \sqrt{t}
\le \|u(t)\|_{L^2}
\le C\|u_0\|_{L^2} + C\sqrt{t}\,\|u_1\|_{L^1}
    + C\|u_1\|_{L^2}
\end{equation}
for $n=1$, and
\begin{equation} \label{eq:1.5}
C |m_1| \sqrt{\log t}
\le \|u(t)\|_{L^2}
\le C\|u_0\|_{L^2} + C\sqrt{\log t}\,\|u_1\|_{L^1}
    + C\|u_1\|_{L^2}
\end{equation}
for $n=2$. On the other hand, in higher dimensions $n\ge3$, the parabolic
character of the equation becomes dominant, and the solution exhibits the
diffusive decay
\begin{equation} \label{eq:1.6}
C |m_1| t^{-\frac{n-2}{4}}
\le \|u(t)\|_{L^2}
\le C(1+t)^{-\frac{n}{4}}\|u_0\|_{L^1\cap L^2}
 + C(1+t)^{-\frac{n-2}{4}}\|u_1\|_{L^1\cap L^2}.
\end{equation}
These facts were established in \cite{DR,HZ,Ike2014,IO,IT,ITY,P,S}.

By contrast, for the standard wave equation without damping,
\begin{equation} \label{eq:1.7}
\begin{cases}
\partial_t^2 v - \Delta v = 0,
    & t>0,\quad x\in\mathbb{R}^n, \\
v(0,x)=u_0(x),\quad \partial_t v(0,x)=u_1(x),
    & x\in\mathbb{R}^n,
\end{cases}
\end{equation}
it is known that the $L^2$-norm of the solution remains globally bounded
in time for $n\ge3$, whereas it grows for $n=1$ and $n=2$ in a manner
analogous to \eqref{eq:1.1}. Namely, for $t\gg1$, the solution to
\eqref{eq:1.7} satisfies
\begin{equation} \label{eq:1.8}
C |m_1| \sqrt{t}
\le \|v(t)\|_{L^2}
\le C\|u_0\|_{L^2} + C\sqrt{t}\,\|u_1\|_{L^1}
    + C\|u_1\|_{L^2}
\end{equation}
for $n=1$,
\begin{equation} \label{eq:1.9}
C |m_1| \sqrt{\log t}
\le \|v(t)\|_{L^2}
\le C\|u_0\|_{L^2} + C\sqrt{\log t}\,\|u_1\|_{L^1}
    + C\|u_1\|_{L^2}
\end{equation}
for $n=2$, and
\begin{equation} \label{eq:1.10}
C |m_1|
\le \|v(t)\|_{L^2}
\le C\|u_0\|_{L^2} + C\|u_1\|_{L^1 \cap L^2}
\end{equation}
for $n\ge3$. These results are due to \cite{CI,CT,Ike2023,Ta2026}.

The estimates \eqref{eq:1.8}--\eqref{eq:1.10} suggest that the growth
rates in \eqref{eq:1.4} and \eqref{eq:1.5} reflect the underlying
hyperbolic nature of solutions to \eqref{eq:1.1}. The main purpose of the
present paper is to clarify to what extent this hyperbolic character
persists. To this end, we recall the asymptotic results obtained in
\cite{Ike2014,IO,IT}. There, the solution to \eqref{eq:1.1} is
approximated globally in time by the diffusion wave
\begin{equation*}
G(t,x)
:= \mathcal{F}^{-1}\left[e^{-\frac{\nu}{2} t|\xi|^2}\omega(t,\xi) \right](x),
\qquad
\omega(t,\xi):=\frac{\sin(t|\xi|)}{|\xi|},
\end{equation*}
in the sense that
\begin{equation*}
\|u(t)-m_1G(t)\|_2 = o(d_n(t))
\qquad \text{as } t\to\infty.
\end{equation*}
Here, $\mathcal{F}^{-1}$ denotes the inverse Fourier transform, and
\begin{equation*}
d_n(t) :=
\begin{cases}
\sqrt{t}, & n=1, \\
\sqrt{\log t}, & n=2, \\
t^{-\frac{n-2}{4}}, & n\ge3,
\end{cases}
\end{equation*}
which corresponds to the rates in \eqref{eq:1.4}--\eqref{eq:1.6}. In this
sense, $u(t)$ is asymptotic to $m_1G(t)$ as $t\to\infty$. The key point
behind this fact is that $\|u(t)-G(t)*u_1\|_2$ is of lower order in time
(see \eqref{eq:2.2}-\eqref{eq:2.4} below for the precise estimate), whereas, for
$u_1\in L^1\cap L^2$, one can in general guarantee only
\[
\|G(t)*u_1-m_1G(t)\|_2=o(d_n(t))
\qquad \text{as } t\to\infty.
\]

Motivated by these observations, we ask whether the solution to the
strongly damped wave equation \eqref{eq:1.1} can be approximated, in low
dimensions, by the solution to the corresponding standard wave equation.
In view of the preceding discussion, this amounts to comparing
$G(t)*u_1$ not only with the leading profile $m_1G(t)$, but also with the
free wave evolution generated by the same initial velocity. Let $W(t)$
denote the solution operator for the standard wave equation, defined by
\begin{equation*}
W(t)g
:= \mathcal{F}^{-1}\left[\omega(t,\xi)\widehat{g}(\xi)\right](x),
\end{equation*}
where $\widehat{g}$ is the Fourier transform of $g$. We define the
difference operator
\begin{equation*}
D(t)g := G(t)*g - W(t)g.
\end{equation*}
The question is whether $D(t)g$ becomes negligible in $L^2$ as
$t\to\infty$. Our results reveal a striking contrast between one and two
space dimensions: such an approximation is valid in dimension one,
whereas it fundamentally fails in dimension two.

We now state the main results of this paper. The first theorem shows that,
in one space dimension, $G(t)*g$ is approximated by $W(t)g$ with the
following rate.
\begin{thm} \label{thm:1.1}
Let $n=1$ and $g\in L^1(\mathbb{R})$. Then there exists a constant $C>0$
such that
\begin{equation} \label{eq:1.11}
\|D(t)g\|_{L^2} \le Ct^{1/4}\|g\|_{L^1}
\end{equation}
for all $t\ge0$.
\end{thm}

In sharp contrast, in two space dimensions, the difference $D(t)g$
exhibits logarithmic divergence. This shows that the standard wave
equation cannot serve as an asymptotic profile for \eqref{eq:1.1} in
dimension two.
\begin{thm} \label{thm:1.2}
Let $n=2$ and $g\in L^1(\mathbb{R}^2)$. Suppose that
\[
m_g := \int_{\mathbb{R}^2} g(y)\,dy \ne 0.
\]
Then there exist constants $C>0$ and $T>0$ such that
\begin{equation} \label{eq:1.12}
\|D(t)g\|_{L^2} \ge C|m_g|\sqrt{\log t}
\end{equation}
for all $t\ge T$.
\end{thm}

The proof of Theorem~\ref{thm:1.1} is based on a straightforward scaling argument,
using the fact that $\omega(t,\xi)$ belongs to $L^2$ in one space
dimension. To isolate the hyperbolic nature of the solution, however, it
is essential to use the change of variables associated with the auxiliary
parabolic scaling, namely replacing $\sqrt{t}\,\xi$ by $\xi$.

By contrast, in two space dimensions, $\omega(t,\xi)$ does not belong to
$L^2$. We therefore introduce the auxiliary function
\begin{equation} \label{eq:1.13}
J^{(\beta)}(t,x)
:= \mathcal{F}^{-1}\left[
e^{-\beta |\xi|^2}\bigl(e^{-\frac{\nu}{2}t|\xi|^2}-1\bigr)
\omega(t,\xi)
\right](x),
\end{equation}
where $\beta>0$ is a parameter. This construction is based on the
smoothing method used in \cite{CT,Ike2023,Ta2026}. The key point is the
proposition proved below. As shown in \eqref{eq:3.5}, it allows us to
reduce the lower bound for $\|D(t)g\|_{L^2}$ to a lower bound for
\[
|m_g|\,\|J^{(\beta)}(t)\|_{L^2},
\]
for which the high-frequency contribution $|\xi|\gg1$ can essentially be
ignored. This method does not require the continuity property
$\widehat{g}\in C(\mathbb{R}^2)$, which follows from
$g\in L^1(\mathbb{R}^2)$. This change of perspective is the main
technical contribution of the present paper.

As a consequence of the two theorems above, we obtain the following
description of the global behavior of the difference between the solution
to \eqref{eq:1.1} and the corresponding solution to the standard wave
equation.
\begin{cor} \label{cor:1.3}
{\rm (i)} Let $n=1$, $u_0\in H^1(\mathbb{R})$, and
$u_1\in L^1(\mathbb{R})\cap L^2(\mathbb{R})$. Then the solution $u$ to
\eqref{eq:1.1} satisfies
\begin{equation} \label{eq:1.14}
\|u(t)-W(t)u_1\|_{L^2}
\le C\|u_0\|_{L^2} + Ct^{1/4}\|u_1\|_{L^1}
    + C\|u_1\|_{L^2}
\end{equation}
for all $t\ge2$.\\

\noindent
{\rm (ii)} Let $n=2$, $u_0\in H^1(\mathbb{R}^2)$, and
$u_1\in L^1(\mathbb{R}^2)\cap L^2(\mathbb{R}^2)$. Suppose that $m_1\ne0$.
Then the solution $u$ to \eqref{eq:1.1} satisfies
\begin{equation} \label{eq:1.15}
C|m_1|\sqrt{\log t}
\le \|u(t)-W(t)u_1\|_{L^2}
\le C\|u_0\|_{L^2} + C\sqrt{\log t}\,\|u_1\|_{L^1}
    + C\|u_1\|_{L^2}
\end{equation}
for all sufficiently large $t$.
\end{cor}

In view of part~(i) of Corollary~\ref{cor:1.3} and \eqref{eq:1.4}, the
solution to \eqref{eq:1.1} approaches the solution to the wave equation
\eqref{eq:1.7} in one space dimension. On the other hand, part~(ii) of
Corollary~\ref{cor:1.3}, together with \eqref{eq:1.5}, shows that such an
approximation fails in two space dimensions.

This paper is organized as follows. In Section~2, we fix notation and
recall the representation formula for solutions to \eqref{eq:1.1},
together with known asymptotic results from previous studies. In
Section~3, we prove the main theorems of this paper.

\section{Preliminaries}
In this section we collect the notation and preliminary estimates needed to state and prove the main results rigorously.
For simplicity, we denote $\|\cdot\|_{L^p}$ by $\|\cdot\|_p$ in what follows.

\subsection{Notation and estimates for the solution of~\eqref{eq:1.1}}

We begin by recalling the notation used throughout the paper, together
with the basic representation formula for solutions to~\eqref{eq:1.1}.
The characteristic roots are given by
\begin{equation*}
\lambda_{\pm}(\xi)
=-\nu |\xi|^{2}\pm \sqrt{\nu^{2}|\xi|^{4}-|\xi|^{2}}.
\end{equation*}
Using these roots, the solution can be represented in Fourier space as
\begin{equation*}
\widehat{u}(t,\xi)
=\widehat{K}_0(t,\xi)\widehat{u}_0(\xi)
 +\widehat{K}_1(t,\xi)\widehat{u}_1(\xi),
\end{equation*}
where
\begin{equation*}
\widehat{K}_0(t,\xi)
=\frac{-\lambda_-(\xi)e^{\lambda_+(\xi)t}
        +\lambda_+(\xi)e^{\lambda_-(\xi)t}}
       {\lambda_+(\xi)-\lambda_-(\xi)},
\qquad
\widehat{K}_1(t,\xi)
=\frac{e^{\lambda_+(\xi)t}-e^{\lambda_-(\xi)t}}
       {\lambda_+(\xi)-\lambda_-(\xi)}.
\end{equation*}
Equivalently, in physical space, we obtain
\begin{equation} \label{eq:2.1}
u(t)=K_0(t) \ast u_0 +K_1(t) \ast u_1.
\end{equation}

We shall use the following known estimates for the global behavior of
$K_0(t)$ and $K_1(t)$ in $L^2$.

\begin{lem}[\cite{P, S}]\label{lem:2.1}
Let $n \ge 1$.  Then there exist constants $C>0$ and $c>0$,
independent of $t$ and $g$, such that
\begin{equation} \label{eq:2.2}
\|K_{0}(t)\ast g\|_{2} \le C \|g\|_{2},
\qquad t\ge0
\end{equation}
and
\begin{equation} \label{eq:2.3}
\|K_{1}(t)\ast g\|_{2}
\le C d_{n}(t) \|g\|_{1}+ C e^{-ct} \|g\|_{2},
\qquad t\ge0.
\end{equation}
\end{lem}

Lemma~\ref{lem:2.1} shows that the leading large-time behavior of the
solution is governed by $K_1(t)$.  Moreover, $K_1(t)$ is close to
$G(t)$ in the following sense.

\begin{lem}[\cite{ITY, IT, Ta2022}]\label{lem:2.2}
Let $n \ge 1$.  Then there exist constants $C>0$ and $c>0$,
independent of $t$ and $g$, such that
\begin{equation} \label{eq:2.4}
\|K_{1}(t)\ast g-G(t)\ast g\|_{2}
\le C (1+t)^{-\frac{n}{4}} \|g\|_{1}+ C e^{-ct} \|g\|_{2},
\qquad t\ge0.
\end{equation}
\end{lem}

A direct proof of the case $n=2$ is given in~\cite{Ta2022}, while the
case $n\ge3$ is treated in~\cite{IT}.  The one-dimensional case follows
by the same argument.

\subsection{Estimates for $J^{(\beta)}(t,x)$}

We next restrict ourselves to two space dimensions and study the global
behavior of the auxiliary function $J^{(\beta)}(t,x)$.

\begin{lem}\label{lem:2.3}
Let $n=2$ and $\beta >0$.  Then, for all $t\ge0$,
\begin{equation} \label{eq:2.5}
\| J^{(\beta)}(t)\|_{2} \le C_{\beta} \sqrt{\log (t+e)},
\qquad
\| \nabla J^{(\beta)}(t)\|_{2} \le C_{\beta},
\end{equation}
where $J^{(\beta)}(t)$ is defined by~\eqref{eq:1.13}.
\end{lem}

\begin{proof}
Using Plancherel's theorem and
passing to polar coordinates in $\mathbb{R}^2$, we obtain
\begin{align*}
\|J^{(\beta)}(t)\|_2^2
&\le C \int_0^\infty e^{-2\beta r^2}
\bigl|e^{-\nu tr^2}-1\bigr|^2\frac{\sin^2(tr)}{r}\,dr.
\end{align*}

Assume first that $t\ge1$ and split the integral at $r=t^{-1}$.  For
$0\le r\le t^{-1}$, using
$|e^{-\nu tr^2}-1|\lesssim tr^2$ and $|\sin(tr)|\lesssim tr$, we get
\[
e^{-2\beta r^2}\bigl|e^{-\nu tr^2}-1\bigr|^2
\frac{\sin^2(tr)}{r}
\lesssim t^4 r^5.
\]
Hence
\[
J_{1}:=
\int_0^{t^{-1}} e^{-2\beta r^2}
\bigl|e^{-\nu tr^2}-1\bigr|^2\frac{\sin^2(tr)}{r}\,dr
\lesssim t^4 \int_0^{t^{-1}} r^5\,dr \lesssim 1.
\]
For $r\ge t^{-1}$, we simply use
$|e^{-\nu tr^2}-1|\le 1$ and $|\sin(tr)|\le 1$ to obtain
\[
J_{2}:=\int_{t^{-1}}^\infty e^{-2\beta r^2}
\bigl|e^{-\nu tr^2}-1\bigr|^2\frac{\sin^2(tr)}{r}\,dr 
\le C \int_{t^{-1}}^1 \frac{dr}{r}
 + C \int_1^\infty e^{-2\beta r^2}\,dr
\le C+ C \log t.
\]
Therefore, for $t\ge1$,
\[
\|J^{(\beta)}(t)\|_2^2 \le J_{1}+J_{2} \le C+ C \log t.
\]

If $0\le t\le1$, then $|e^{-\nu tr^2}-1|\le1$, and hence
\begin{equation*}
\begin{split}
\|J^{(\beta)}(t)\|_2^2
& \le C \int_0^\infty e^{-2\beta r^2}\frac{\sin^2(tr)}{r}\,dr \\
& \le Ct^2\int_0^1  e^{-2\beta r^2} \left| \frac{\sin(tr)}{tr} \right|^{2} r\,dr
 + C\int_1^\infty e^{-2\beta r^2}\,dr \\
& \le Ct^2\int_0^1 r\,dr
 + C\int_1^\infty e^{-2\beta r^2}\,dr
\le C.
\end{split}
\end{equation*}
Combining the two cases yields
\[
\|J^{(\beta)}(t)\|_2^2 \le C \log(t+e),
\]
and we arrive at the desired estimate 
\[
\|J^{(\beta)}(t)\|_2
\le C_\beta \sqrt{\log(t+e)}.
\]

For the gradient, Plancherel's theorem gives
\[
\|\nabla J^{(\beta)}(t)\|_2^2
\le C \int_{\mathbb{R}^2}
|\xi|^2 e^{-2\beta |\xi|^2}
\bigl|e^{-\nu t|\xi|^2}-1\bigr|^2
\omega(t,\xi)^{2}
\,d\xi \le C\int_{\mathbb{R}^2} e^{-2\beta |\xi|^2}\,d\xi
<\infty.
\]
This proves~\eqref{eq:2.5}.
\end{proof}

The following proposition describes the asymptotic behavior of
$J^{(\beta)}(t)\ast g$ and is essential in the proof of
Theorem~\ref{thm:1.2}.

\begin{prop}\label{prop:2.4}
Let $n=2$, $\beta >0$ and $g \in L^{1}(\mathbb{R}^{2})$.  Then
\begin{equation} \label{eq:2.6}
\| J^{(\beta)}(t) \ast g-m_{g} J^{(\beta)}(t) \|_{2}
=o\bigl(\sqrt{\log t}\bigr)
\end{equation}
as $t \to \infty$.
\end{prop}

\begin{proof}
Put $L_t:=(\log(t+e))^{1/4}$ for the simplicity.  Since
\[
J^{(\beta)}(t)\ast g-m_g J^{(\beta)}(t)
=\int_{\mathbb{R}^2}
\bigl(J^{(\beta)}(t,\cdot-y)-J^{(\beta)}(t,\cdot)\bigr)g(y)\,dy,
\]
we decompose this integral as $I_1(t)+I_2(t)$, where
\begin{align*}
I_1(t,x)&:=\int_{|y|\le L_t}
\bigl(J^{(\beta)}(t,x-y)-J^{(\beta)}(t,x)\bigr)g(y)\,dy,\\
I_2(t,x)&:=\int_{|y|>L_t}
\bigl(J^{(\beta)}(t,x-y)-J^{(\beta)}(t,x)\bigr)g(y)\,dy.
\end{align*}

We first estimate $I_1$.  By the mean value theorem,
\[
|J^{(\beta)}(t,x-y)-J^{(\beta)}(t,x)|
\le |y|\int_0^1
|\nabla_x J^{(\beta)}(t,x-\theta y)|\,d\theta.
\]
Therefore, by Minkowski's inequality and \eqref{eq:2.5},
\begin{align*}
\|I_1(t)\|_2
&\le \int_{|y|\le L_t}|y|\,|g(y)|
\int_0^1\|\nabla_x J^{(\beta)}(t,\cdot-\theta y)\|_2\,d\theta\,dy\\
&\le C_\beta L_t\|g\|_1
=O\bigl((\log(t+e))^{1/4}\bigr)
=o\bigl(\sqrt{\log t}\bigr).
\end{align*}

Next, using \eqref{eq:2.5}, we have
\begin{align*}
\|I_2(t)\|_2
&\le \int_{|y|>L_t}
\|J^{(\beta)}(t,\cdot-y)-J^{(\beta)}(t,\cdot)\|_2 |g(y)|\,dy\\
&\le 2\|J^{(\beta)}(t)\|_2
\int_{|y|>L_t}|g(y)|\,dy\\
&\le C_\beta\sqrt{\log(t+e)}
\int_{|y|>L_t}|g(y)|\,dy
=o\bigl(\sqrt{\log(t+e)}\bigr),
\end{align*}
because $L_t\to\infty$ and $g\in L^1(\mathbb{R}^2)$.

Combining the estimates for $I_1(t)$ and $I_2(t)$ proves the claim.
\end{proof}

\section{Proof of main results}

In this section we prove Theorems~\ref{thm:1.1}--\ref{thm:1.2} and
Corollary~\ref{cor:1.3}.

\subsection{Proof of Theorem \ref{thm:1.1}}
At first, we observe that 
\begin{equation} \label{eq:3.1}
\begin{split}
\| D(t)g \|_{2} 
\le \left\| G(t) - \mathcal{F}^{-1} \left[ \omega(t,\xi) \right] \right\|_{2} \| g\|_{1}.
\end{split}
\end{equation}
To this end, we show the estimate for 
$\left\| \hat{G}(t) - \omega(t,\xi)\right\|_{2}$.
Then changing the variable, we see that 
\begin{equation*}
\begin{split}
\left\| \hat{G}(t) - \omega(t,\xi)\right\|_{2}^{2}
& = \frac{\sqrt{t}}{2\pi} 
\int_{-\infty}^{\infty} \frac{\sin^2(\sqrt{t} \xi)}{\xi^2} (1 - e^{-\nu \xi^2})^2 d\xi \\
& =I_{1}+I_{2},
\end{split}
\end{equation*}
where 
\begin{equation*}
\begin{split}
I_{1} & := \frac{\sqrt{t}}{2\pi} 
\int_{|\xi| \le 1} \frac{\sin^2(\sqrt{t} \xi)}{\xi^2} (1 - e^{-\nu \xi^2})^2 d\xi, \\
I_{2} & := \frac{\sqrt{t}}{2\pi} 
\int_{|\xi| \ge 1} \frac{\sin^2(\sqrt{t} \xi)}{\xi^2} (1 - e^{-\nu \xi^2})^2 d\xi.
\end{split}
\end{equation*}

For $I_{1}$, 
applying the mean value theorem to have 
\begin{equation*}
\begin{split}
1 - e^{-\nu \xi^2} = \nu \xi^{2}  e^{-\nu \theta \xi^2} 
\end{split}
\end{equation*}
for $\theta \in (0,1)$, 
we see that 
\begin{equation} \label{eq:3.2}
\begin{split}
I_{1} & \le C \sqrt{t} 
\int_{|\xi| \le 1}  \xi^2 d\xi \le C \sqrt{t}.
\end{split}
\end{equation}
On the other hand, 
we also have 
\begin{equation} \label{eq:3.3}
\begin{split}
I_{2} & \le C \sqrt{t}
\int_{|\xi| \ge 1} \frac{1}{\xi^2} d\xi \le C \sqrt{t}.
\end{split}
\end{equation}
Therefore we conclude that 
\begin{equation} \label{eq:3.4}
\begin{split}
I_{1}+I_{2} \le \frac{\sin^2(t \xi)}{\xi^2} (1 - e^{-\nu t \xi^2})^2 d\xi \le C \sqrt{t}
\end{split}
\end{equation}
by \eqref{eq:3.2} and \eqref{eq:3.3}.
Summing up \eqref{eq:3.1} and \eqref{eq:3.4},
we obtain the desired estimate \eqref{eq:1.8}.
\subsection{Proof of Theorem \ref{thm:1.2}}
We follow the smoothing argument used in
\cite{CI, CT, Ike2023, Ta2026}.  The purpose of introducing the
factor $e^{-\beta|\xi|^{2}}$ is to remove the high-frequency part of
the multiplier without changing the logarithmic growth coming from the
low-frequency region.  By Plancherel's theorem and the definition of
$D(t)$, we have
\begin{equation} \label{eq:3.5}
\begin{split}
\|D(t)g\|_{2}
&= \|\widehat{D(t)g}\|_{2}  \\
&\ge \|e^{-\beta|\xi|^{2}}\widehat{D(t)g}\|_{2} = \|J^{(\beta)}(t)\ast g\|_{2} \\
&\ge |m_g|\|J^{(\beta)}(t)\|_{2}
   -\|J^{(\beta)}(t)\ast g-m_gJ^{(\beta)}(t)\|_{2}.
\end{split}
\end{equation}
Therefore,
by \eqref{eq:2.5}, it remains only to
derive a lower bound for $\|J^{(\beta)}(t)\|_{2}$ of order
$\sqrt{\log t}$.

We assume, without loss of generality, that $t\ge4$.  By Plancherel's
theorem,
\begin{equation*}
\begin{split}
\|J^{(\beta)}(t)\|_{2}^{2}
&\ge \int_{t^{-1/2}\le|\xi|\le1}
e^{-2\beta|\xi|^{2}}
\left|e^{-\nu t|\xi|^{2}}-1\right|^{2}
|\omega(t,\xi)|^{2} \,d\xi.
\end{split}
\end{equation*}
Using
\[
\sin^{2}(t|\xi|)=\frac{1-\cos(2t|\xi|)}{2},
\]
we decompose the last integral as
\[
\|J^{(\beta)}(t)\|_{2}^{2}\ge A_{1}-A_{2},
\]
where
\begin{equation*}
\begin{split}
A_{1}
&:= \int_{t^{-1/2}\le|\xi|\le1}
\frac{
e^{-2\beta|\xi|^{2}}
\left|e^{-\nu t|\xi|^{2}}-1\right|^{2}}
{2|\xi|^{2}}\,d\xi, \\
A_{2}
&:= \int_{t^{-1/2}\le|\xi|\le1}
e^{-2\beta|\xi|^{2}}
\left|e^{-\nu t|\xi|^{2}}-1\right|^{2}
\frac{\cos(2t|\xi|)}{2|\xi|^{2}}\,d\xi .
\end{split}
\end{equation*}

We first estimate the non-oscillatory term $A_{1}$.  Passing to polar
coordinates, and using $tr^{2}\ge1$ on
$t^{-1/2}\le r\le1$, we obtain
\begin{equation*}
\begin{split}
A_{1}
&= \pi \int_{t^{-1/2}}^{1}
e^{-2\beta r^{2}}
\left|e^{-\nu tr^{2}}-1\right|^{2}\frac{dr}{r} \\
&\ge \pi e^{-2\beta}(1-e^{-\nu})^{2}
\int_{t^{-1/2}}^{1}\frac{dr}{r} \\
&= \frac{\pi}{2}e^{-2\beta}(1-e^{-\nu})^{2}\log t .
\end{split}
\end{equation*}
Thus $A_{1}$ gives the desired logarithmic growth.

It remains to show that the oscillatory term $A_{2}$ is bounded.  In
polar coordinates,
\[
A_{2}
= \pi \int_{t^{-1/2}}^{1}
e^{-2\beta r^{2}}
\left(e^{-\nu tr^{2}}-1\right)^{2}
\frac{\cos(2tr)}{r}\,dr .
\]
Set
\[
F_t(r):=
e^{-2\beta r^{2}}
\left(e^{-\nu tr^{2}}-1\right)^{2}
\frac{1}{r}.
\]
Then integration by parts gives
\begin{equation*}
A_{2}
= \frac{\pi}{2t}
\left[F_t(r)\sin(2tr)\right]_{r=t^{-1/2}}^{1}
-\frac{\pi}{2t}
\int_{t^{-1/2}}^{1}F_t'(r)\sin(2tr)\,dr .
\end{equation*}
The boundary term is $O(t^{-1/2})$, because
$F_t(1)=O(1)$ and $F_t(t^{-1/2})=O(t^{1/2})$.

We next estimate the integral involving $F_t'$.  A direct computation
shows that
\begin{equation*}
\begin{split}
F_t'(r)
={}&
-4\beta e^{-2\beta r^{2}}
\left(e^{-\nu tr^{2}}-1\right)^{2}  \\
&-4\nu t e^{-2\beta r^{2}}e^{-\nu tr^{2}}
\left(e^{-\nu tr^{2}}-1\right)  \\
&-e^{-2\beta r^{2}}
\left(e^{-\nu tr^{2}}-1\right)^{2}\frac{1}{r^{2}} .
\end{split}
\end{equation*}
Therefore,
\begin{equation*}
\begin{split}
\frac{1}{t}\int_{t^{-1/2}}^{1}
\left(e^{-\nu tr^{2}}-1\right)^{2}\,dr
&=O(t^{-1}), \\
\int_{t^{-1/2}}^{1}
e^{-\nu tr^{2}}
\left|e^{-\nu tr^{2}}-1\right|\,dr
&=O(t^{-1/2}), \\
\frac{1}{t}\int_{t^{-1/2}}^{1}\frac{dr}{r^{2}}
&=O(t^{-1/2}).
\end{split}
\end{equation*}
Consequently,
\[
A_{2}=O(1)
\qquad \text{as } t\to\infty .
\]
Combining the estimates for $A_{1}$ and $A_{2}$, we obtain
\[
\|J^{(\beta)}(t)\|_{2}^{2}
\ge C\log t-C
\]
for all sufficiently large $t$.  Hence,
\[
\|J^{(\beta)}(t)\|_{2}\ge C\sqrt{\log t}
\]
for sufficiently large $t$.

Finally, returning to \eqref{eq:3.5} and using \eqref{eq:2.6},
we conclude that
\[
\|D(t)g\|_{2}
\ge C|m_g|\sqrt{\log t}
\]
as $t\to\infty$.  This proves Theorem~\ref{thm:1.2}.

\subsection{Proof of Corollrary \ref{cor:1.3}}

Using the representation formula~\eqref{eq:2.1}, we decompose
\begin{equation} \label{eq:3.6}
u(t)-W(t)u_{1}
=K_{0}(t)\ast u_{0}
 +(K_{1}(t)-G(t))\ast u_{1}
 +D(t)u_{1}.
\end{equation}
In the case $n=1$, a direct estimate of~\eqref{eq:3.6}, based on
\eqref{eq:1.11}, \eqref{eq:2.2}, and~\eqref{eq:2.4}, yields
\eqref{eq:1.14}.

We next consider the case $n=2$. The corresponding upper bound is
already known from~\eqref{eq:1.5} and~\eqref{eq:1.9}; hence it remains
only to prove the lower bound in~\eqref{eq:1.15}. Since $m_{1}\ne0$,
\eqref{eq:1.12} implies that there exist constants $c>0$ and $T>0$
such that
\begin{equation} \label{eq:3.7}
\|D(t)u_{1}\|_{2}
\ge c |m_{1}|\sqrt{\log t},
\qquad t\ge T .
\end{equation}
On the other hand, \eqref{eq:2.2} and~\eqref{eq:2.5} give
\[
\|K_{0}(t)\ast u_{0}\|_{2}
+
\|(K_{1}(t)-G(t))\ast u_{1}\|_{2}
=o(\sqrt{\log t})
\]
as $t\to\infty$. Therefore, applying the triangle inequality to
\eqref{eq:3.6}, we obtain
\[
\|u(t)-W(t)u_{1}\|_{2}
\ge \|D(t)u_{1}\|_{2}
 -\|K_{0}(t)\ast u_{0}\|_{2}
 -\|(K_{1}(t)-G(t))\ast u_{1}\|_{2}.
\]
Combining this inequality with~\eqref{eq:3.7}, and decreasing $c>0$ if
necessary, we conclude that
\[
\|u(t)-W(t)u_{1}\|_{2}
\ge c|m_{1}|\sqrt{\log t}
\]
for sufficiently large $t$. This proves
Corollary~\ref{cor:1.3}.

\vspace*{5mm}
\noindent
\textbf{Acknowledgments. }
\smallskip
The second author was partially supported by
JSPS KAKENHI Grant Number 24K06822.

\end{document}